\documentclass[11pt]{amsart}
\usepackage{amsmath}
\usepackage[active]{srcltx}
\usepackage{t1enc}
\usepackage[latin2]{inputenc}
\usepackage{verbatim}
\usepackage{amsmath,amsfonts,amssymb,amsthm}
\usepackage[mathcal]{eucal}
\usepackage{enumerate}
\usepackage[centertags]{amsmath}
\usepackage{graphics}
\usepackage[active]{srcltx}

\newtheorem{theorem}{Theorem}
\newtheorem{lemma}{Lemma}

\newtheorem*{OldTheorem}{Theorem A}
\newtheorem*{OldTheorem2}{Theorem B}
\newtheorem*{OldTheorem3}{Theorem C}
\newtheorem*{OldTheorem4}{Theorem D}

\def\bmo{{\rm BMO\,}}

\def\ZR{\ensuremath{\mathbb R}}

\def\ZT{\ensuremath{\mathbb T}}
\def\ZI{\ensuremath{\mathbb I}}
\def\ZB{\ensuremath{\mathcal B}}

\def\md#1#2\emd{
\ifx0#1\begin{equation*} #2 \end{equation*}\fi
\ifx1#1\begin{equation}#2\end{equation}\fi
\ifx2#1\begin{align*}#2\end{align*}\fi
\ifx3#1\begin{align}#2\end{align}\fi
\ifx4#1\begin{gather*}#2\end{gather*}\fi
\ifx5#1\begin{gather}#2\end{gather}\fi
\ifx6#1\begin{multline*}#2\end{multline*}\fi
\ifx7#1\begin{multline}#2\end{multline}\fi
\ifx8#1\begin{multline*}\begin{split}#2\end{split}\end{multline*}\fi
\ifx9#1\begin{multline}\begin{split}#2\end{split}\end{multline}\fi
}
\newcommand {\e }[1]{(\ref{#1})}

\newcommand {\trm }[1]{Theorem \ref{#1}}

\begin{document}

\author{U. Goginava, L. Gogoladze and G. Karagulyan}
\title[On the Exponential Almost Everywhere Summability]{BMO-estimation and Almost Everywhere Exponential Summability of
Quadratic Partial Sums of Double Fourier Series%
}
\address{U. Goginava, Department of Mathematics, Faculty of Exact and Natural
Sciences, Iv. Javakhishvili Tbilisi State University, Chavcha\-vadze str. 1, Tbilisi 0128,
Georgia}
\email{zazagoginava@gmail.com}

\address{L. Gogoladze, Department of Mathematics, Faculty of Exact and Natural
Sciences, Iv. Javakhishvili Tbilisi State University, Chavcha\-vadze str. 1, Tbilisi 0128,
Georgia}
\email{lgogoladze1@hotmail.com}

\address{G. Karagulyan, Institue of Mathematics of Armenian National Academy
of Science, Bughramian Ave. 24/5, 375019, Yerevan, Armenia}
\email{g.karagulyan@yahoo.com}

\maketitle

\begin{abstract}
It is proved a $\bmo$-estimation for quadratic partial sums of two-dimensional
Fourier series from which it is derived an almost everywhere exponential
summability of quadratic partial sums of double Fourier series.
\end{abstract}

\medskip

\footnotetext{%

\par
The research of U. Goginava was supported by Shota Rustaveli National
Science Foundation grant no.31/48 (Operators in some function spaces and their applications in
Fourier analysis)}

\section{Introduction}

Let $\mathbb{T}:=[-\pi ,\pi )=\ZR/2\pi$ and $\mathbb{R}:=\left( -\infty ,\infty
\right) $. We denote by $L_{1}\left( \mathbb{T}\right) $ the class of all
measurable functions $f$ on $\mathbb{R}$ that are $2\pi $-periodic and
satisfy
\begin{equation*}
\left\Vert f\right\Vert _{1}:=\int\limits_{\mathbb{T}}\left\vert
f\right\vert <\infty .
\end{equation*}%
The Fourier series of the function $f\in L_{1}\left( \mathbb{T}\right) $
with respect to the trigonometric system is the series%
\begin{equation}
\sum_{n=-\infty }^{\infty }\widehat{f}\left( n\right) e^{inx},
\label{fourier}
\end{equation}%
where
\begin{equation*}
\widehat{f}\left( n\right) :=\frac{1}{2\pi }\int\limits_{\mathbb{T}}f\left(
x\right) e^{-inx}dx
\end{equation*}%
are the Fourier coefficients of $f$.

Denote by $S_{n}(x,f)$ the partial sums of the Fourier series of $f$ and let
\begin{equation*}
\sigma _{n}(x,f)=\frac{1}{n+1}\sum_{k=0}^{n}S_{k}(x,f)
\end{equation*}%
be the $(C,1)$ means of (\ref{fourier}). Fej\'er \cite{Fe} proved that $\sigma
_{n}(f)$ converges to $f$ uniformly for any $2\pi $-periodic continuous
function. Lebesgue in \cite{Le} established almost everywhere convergence of
$(C,1)$ means if $f\in L_{1}(\mathbb{T})$. The strong summability problem,
i.e. the convergence of the strong means
\begin{equation}
\frac{1}{n+1}\sum\limits_{k=0}^{n}\left\vert S_{k}\left( x,f\right) -f\left(
x\right) \right\vert ^{p},\quad x\in \mathbb{T},\quad p>0,  \label{Hp}
\end{equation}%
was first considered by Hardy and Littlewood in \cite{H-L}. They showed that
for any $f\in L_{r}(\mathbb{T})~\left( 1<r<\infty \right) $ the strong means
tend to $0$ a.e., if $n\rightarrow \infty $. The trigonometric Fourier
series of $f\in L_{1}(\mathbb{T})$ is said to be $\left( H,p\right) $%
-summable at $x\in T$, if the values \e{Hp} converge to $0$ as $n\rightarrow
\infty $. The $\left( H,p\right) $-summability problem in $L_{1}(\mathbb{T})$
has been investigated by Marcinkiewicz \cite{Ma} for $p=2$, and later by
Zygmund \cite{Zy2} for the general case $1\leq p<\infty $. K.~I.~Oskolkov in
\cite{Os} proved the following

\begin{OldTheorem}
Let $f\in L_{1}(\mathbb{T})$ and let $\Phi $ be a continuous positive convex
function on $[0,+\infty )$ with $\Phi \left( 0\right) =0$ and
\begin{equation}  \label{a1}
\ln \Phi \left( t\right) =O\left( t/\ln \ln t\right) \text{ \ \ \ }\left(
t\rightarrow \infty \right) .
\end{equation}
Then for almost all $x$%
\begin{equation}
\lim\limits_{n\rightarrow \infty }\frac{1}{n+1}\sum\limits_{k=0}^{n}\Phi
\left( \left\vert S_{k}\left( x,f\right) -f\left( x\right) \right\vert
\right) =0.  \label{osk}
\end{equation}
\end{OldTheorem}

It was noted in \cite{Os} that V.~Totik announced the conjecture that (\ref%
{osk}) holds almost everywhere for any $f\in L_{1}(\mathbb{T})$, provided
\begin{equation}  \label{a2}
\ln \Phi \left( t\right) =O\left( t\right) \quad \left( t\rightarrow \infty
\right).
\end{equation}
In \cite{Ro} V.Rodin proved

\begin{OldTheorem2}
Let $f\in L_{1}(\mathbb{T})$. Then for any $A>0$
\begin{equation*}
\lim\limits_{n\rightarrow \infty }\frac{1}{n+1}\sum\limits_{k=0}^{n}\left(
\exp \left( A\left\vert S_{k}\left( x,f\right) -f\left( x\right) \right\vert
\right) -1\right) =0
\end{equation*}%
for a. e. $x\in \mathbb{T}$.
\end{OldTheorem2}

G.~Karagulyan \cite{Ka} proved that the following is true.

\begin{OldTheorem3}
Suppose that a continuous increasing function $\Phi :[0,\infty )\rightarrow
\lbrack 0,\infty ),\Phi \left( 0\right) =0$, satisfies the condition%
\begin{equation*}
\limsup_{t\rightarrow +\infty }\frac{\log \Phi \left( t\right) }{t}=\infty .
\end{equation*}%
Then there exists a function $f\in L_{1}(\mathbb{T})$ for which the relation%
\begin{equation*}
\limsup_{n\rightarrow \infty }\frac{1}{n+1}\sum\limits_{k=0}^{n}\Phi \left(
\left\vert S_{k}\left( x,f\right) \right\vert \right) =\infty
\end{equation*}%
holds everywhere on $\mathbb{T}$.
\end{OldTheorem3}

In fact, Rodin in \cite{Ro} has obtained a $\bmo$ estimate for the partial
sums of Fourier series and his theorem stated above is obtained from that
estimate by using John-Nirenberg theorem. Recall the definition of $\bmo[0,1]$
space. It is the Banach space of functions $f\in L_{1}[0,1]$ with the norm

\begin{equation*}
\Vert f\Vert _{\bmo}=\mathfrak{R}(f)+\left\vert \int_{0}^{1}f(t)dt\right\vert
\end{equation*}%
where
\begin{equation*}
\mathfrak{R}(f)=\sup_{I}(|f-f_{I}|)_{I}, f_{I}=\frac{1}{|I|}\int_{I}f(t)dt
\end{equation*}%
and the supremum is taken over all intervals $I\subset \lbrack 0,1]$ (\cite%
{Gar}, chap. 6). Let $\{\xi _{n}:\,n=0,1,2,\ldots \}$ be an arbitrary
sequence of numbers. Taking $\delta _{k}^{n}=[k/{(n+1)},(k+1)/{(n+1)}]$, we
define

\begin{equation*}
\bmo\left[ \xi _{n}\right] =\sup\limits_{0\leq n<\infty }\left\Vert
\sum\limits_{k=0}^{n}\xi _{k}\mathbb{I}_{\delta _{k}^{n}}\left( t\right)
\right\Vert _{\bmo}
\end{equation*}%
where $\mathbb{I}_{\delta _{k}^{n}}(t)$ is the characteristic function of $%
\delta _{k}^{n}$. Notice that the expressions
\begin{equation}
\bmo\left[ \widetilde{S}_{n}\left( x,f\right) \right] ,~\ \ \bmo\left[
S_{n}\left( x,f\right) \right] ,~\ f\in L_{1}\left( \mathbb{T}\right)
,x\in \mathbb{T}  \label{e}
\end{equation}%
define a sublinear operators, where $\widetilde{S}_{n}\left( x,f\right) $ is
the conjugate partial sum. The following theorem is proved by Rodin in \cite%
{Ro}.

\begin{OldTheorem4}
The operators (\ref{e}) are of weak type $(1,1)$, i.e. the inequalities
\begin{equation}\label{bmo1}
|\{x\in \mathbb{T}:\,\bmo\mathrm{\,}[S_{n}(x,f)]>\lambda \}|\leq \frac{c}{%
\lambda }\int_{\mathbb{T}}|f(t)|dt
\end{equation}
and
\begin{equation}\label{bmo2}
|\{x\in \mathbb{T}:\,\bmo\mathrm{\,}[\widetilde{S}_{n}(x,f)]>\lambda \}|\leq
\frac{c}{\lambda }\int_{\mathbb{T}}|f(t)|dt
\end{equation}%
hold for any $f\in L_{1}(\mathbb{T})$.
\end{OldTheorem4}
In this paper we study the question of exponential summability of
quadratic partial sums of double Fourier series. Let $f\in L_{1}(\mathbb{T}%
^{2})$, be a function
with Fourier series
\begin{equation}
\sum_{m,n=-\infty }^{\infty }\widehat{f}\left( m,n\right) e^{i(mx+ny)},
\label{DF}
\end{equation}%
where
\begin{equation*}
\widehat{f}\left( m,n\right) =\frac{1}{4\pi ^{2}}\iint\limits_{\mathbb{T}%
^{2}}f(x,y)e^{-i(mx+ny)}dxdy
\end{equation*}%
are the Fourier coefficients of the function $f$. The rectangular partial
sums of (\ref{DF}) are defined as follows:
\begin{equation*}
S_{MN}\left( x,y,f\right) =\sum_{m=-M}^{M}\sum_{n=-N}^{N}\widehat{f}\left(
m,n\right) e^{i(mx+ny)}.
\end{equation*}%
We denote by $L\log L\left( \mathbb{T}^{2}\right) $ the class of measurable
functions $f$, with
\begin{equation*}
\iint\limits_{\mathbb{T}^{2}}|f|\log ^{+}|f|<\infty ,
\end{equation*}%
where $\log ^{+}u:=\mathbb{I}_{(1,\infty )}\log u.$ For quadratic partial sums
of two-dimensional trigonometric Fourier series Marcinkiewicz \cite{Ma2} has
proved, that if $f\in L\log L\left( \mathbb{T}^{2}\right) $, then%
\begin{equation*}
\lim\limits_{n\rightarrow \infty }\frac{1}{n+1}\sum\limits_{k=0}^{n}\left(
S_{kk}\left( x,y,f\right) -f\left( x,y\right) \right) =0
\end{equation*}%
for a. e. $\left( x,y\right) \in \mathbb{T}^{2}$. L.~Zhizhiashvili \cite{Zh}
improved this result showing that class $L\log L\left( \mathbb{T}^{2}\right)
$ can be replaced by $L_{1}\left( \mathbb{T}^{2}\right) $.

From a result of S.~Konyagin \cite{Kon} it follows that for every $%
\varepsilon >0$ there exists a function $f\in L\log ^{1-\varepsilon }\left(
\mathbb{T}^{2}\right) $ such that
\begin{equation}
\lim\limits_{n\rightarrow \infty }\frac{1}{n+1}\sum\limits_{k=0}^{n}\left%
\vert S_{kk}\left( x,y,f\right) -f\left( x,y\right) \right\vert \neq 0\text{
\ \ for a. e. }\left( x,y\right) \in \mathbb{T}^{2}.  \label{str}
\end{equation}

The main result of the present paper is the following.

\begin{theorem}\label{BMO}
If $f\in L\log L\left( \mathbb{T}^{2}\right) $,
then
\begin{eqnarray}
&&\left\vert \{(x,y)\in \mathbb{T}^{2}:\bmo[S_{nn}(f,x,y)]>\lambda \}\right\vert \\
\newline
&\leq &\frac{c}{\lambda }\left( 1+\iint\limits_{\mathbb{T}^{2}}|f|\log
^{+}|f)|\right)  \notag
\end{eqnarray}
for any $\lambda >0$, where $c$ is an absolute positive constant.
\end{theorem}

The following theorem shows that the quadratic sums of two-dimensional
Fourier series of a function $f\in L\log L\left( \mathbb{T}^{2}\right) $ are
almost everywhere exponentially summable to the function $f$. It will be
obtained from the previous theorem by using John-Nirenberg theorem.

\begin{theorem}
\label{exp} Suppose that $f\in L\log L\left( \mathbb{T}^{2}\right) $. Then
for any $A>0$%
\begin{equation*}
\lim\limits_{m\rightarrow \infty }\frac{1}{m+1}\sum\limits_{n=0}^{m}\left(
\exp \left( A\left\vert S_{nn}\left( x,y,f\right) -f\left( x,y\right)
\right\vert \right) -1\right) =0
\end{equation*}%
for a. e. $\left( x,y\right) \in \mathbb{T}^{2}$.
\end{theorem}

According to a Lemma of L. D. Gogoladze \cite{Gog}, this theorem can be
formulated in more general settings.

\begin{theorem}
Let $\psi:[0,\infty)\to [0,\infty)$ be a increasing function,  satisfying the conditions
\md4
\lim_{u\rightarrow 0 }\psi \left( u\right) =\psi \left( 0\right)=0,\newline
\limsup_{u\rightarrow \infty }\frac{\log \psi \left( u\right) }{u}<\infty .
\emd
Then for any $f\in L\log L\left( \mathbb{T} ^{2}\right)$ we have
\begin{equation*}
\lim\limits_{m\rightarrow \infty }\frac{1}{m+1}\sum\limits_{n=0}^{m}\psi
\left( \left\vert S_{nn}\left( x,y,f\right) -f\left( x,y\right) \right\vert
\right) =0
\end{equation*}
almost everywhere on $\mathbb{T} ^{2}$.
\end{theorem}

The results on Marcinkiewicz type strong summation for the Fourier series have
been investigated in \cite{FS,Ga,Gog2,Gl,GG,GGT,Gogi,lei,tot,Wa,We,Zh}

\section{Notations and lemmas}
The relation $a\lesssim b$ bellow stands for $a\leq c\cdot b$, where
$c$ is an absolute constant. The conjugate function of a given $f\in L_1(\ZT)$ is defined by
\md0
\tilde f(x)=\text{p.v.}\frac{1}{\pi }\int\limits_\ZT \frac{f(x+t)}{2\tan(t/2)}dt=\lim_{\varepsilon\to 0}\frac{1}{\pi }\int\limits_{\varepsilon<|t|<\pi} \frac{f(x+t)}{2\tan(t/2)}dt.
\emd
According to Kolmogorov's and Zygmund's inequalities (see \cite{Zy2}, chap. 7), we have%
\md3
|\{x\in \ZT:\, |\tilde f(x)|>\lambda\}|\lesssim \frac{\|f\|_{L_1(\ZT)}}{\lambda},\label{Kolmogorov}\\
\int\limits_{\ZT}|\tilde f(x)|dx\lesssim 1+\int\limits_{\mathbb{T}}|f(x)|\log ^{+}|f(x)|dx.
\label{Zygmund}
\emd
It will be used two simple properties of $\bmo$ norm below. First one says, if $\xi_n=c$, $n=1,2,\ldots $, then $\bmo\left[ \xi _{n}\right]=|c|$. The second one is, the bound
\md0
\bmo[\xi_n]\le 3\sup_n|\xi_n|.
\emd
We  shall consider the operators
\md0
U_n(x,f)=\text{p.v.}\frac{1}{\pi}\int\limits_\ZT\frac{\cos nt}{2\tan( t/2)}f(x+t)dt .
\emd
The following lemma is an immediate consequence of Theorem D.
\begin{lemma}
The inequality
\md0
|\{x\in \ZT:\, \bmo[U_n(x,f)]>\lambda\}| \lesssim \frac{\|f\|_{L_1(\ZT)}}{\lambda}
\emd
holds for any $f\in L_1(\ZT)$.
\end{lemma}
\begin{proof}
For the conjugate Dirichet kernel we have
\md9\label{Dir}
\tilde D_n(t)&=\frac{\cos(t/2)-\cos(n+1/2)t}{2\sin(t/2)}\\
&=\frac{1}{2\tan(t/2)}+\frac{\sin nt}{2}-\frac{\cos nt}{2\tan(t/2)}
\emd
and we get
\md8
\tilde S_n(x,f)&=\frac{1}{\pi}\int\limits_\ZT\tilde D_n(t)f(x+t)dt\\
&=\tilde f(x)+\frac{1}{2\pi}\int\limits_\ZT f(x+t)\sin ntdt-U_n(x,f).
\emd
Thus, applying simple properties of $\bmo$ norm, we obtain
\md0
\bmo[U_n(x,f)]\le |\tilde f(x)|+\frac{1}{2\pi}\int\limits_\ZT |f(t)|dt+\bmo\left[\tilde S_n(x,f)\right]
\emd
Applying the bound \e{Kolmogorov} and Theorem D, the last inequality completes the proof of lemma.
\end{proof}
We consider the square partial sums
\md1\label{a20}
S_{nn}\left( x,y,f\right) =\frac{1}{\pi ^{2}}\iint\limits_{\mathbb{T}^{2}}%
\frac{\sin \left( n+1/2\right) t\sin \left( n+1/2\right) s}{4\sin \left(
t/2\right) \sin \left( s/2\right) }f\left( x+t,y+s\right) dtds
\emd
and their modification, defined by
\md0
S_{nn}^*\left( x,y,f\right) =\frac{1}{\pi ^{2}}\iint\limits_{\mathbb{T}^{2}}%
\frac{\sin nt \sin ns}{4\tan \left(
t/2\right) \tan \left( s/2\right) }f\left( x+t,y+s\right) dtds.
\emd
\begin{lemma}
If $f\in L\log L(\ZT^2)$, then
\md0
\iint\limits_{\ZT^2}\sup_n\left|S_{nn}(x,y,f)-S_{nn}^*(x,y,f)\right|dxdy\lesssim  1+\iint\limits_{\mathbb{T}^{2}}|f|\log ^{+}|f|.
\emd
\end{lemma}
\begin{proof}
Substituting the expression for Dirichlet kernel
\md0
D_n(t)=\frac{\sin(n+1/2)t}{2\sin t/2}=\frac{\sin n t}{2\tan (t/2)} +\frac{\cos nt}{2}
\emd
in \e{a20}, we get
\md8
S_{nn}\left( x,y,f\right)&-S_{nn}^*(x,y,f)\\
&=\frac{1}{\pi ^{2}}\iint\limits_{\mathbb{T}^{2}}
\frac{\sin n t\cdot\cos n s}{4\tan (t/2)}  f\left( x+t,y+s\right) dtds   \\
&+\frac{1}{\pi ^{2}}\iint\limits_{\mathbb{T}^{2}}\frac{\cos n t\cdot\sin ns}{4\tan (s/2)} f\left( x+t,y+s\right) dtds  \\
&+\frac{1}{4\pi ^{2}}\iint\limits_{\mathbb{T}^{2}}\cos n t\cdot\cos n s\cdot f\left( x+t,y+s\right) dtds  \\
&={S}_{nn}^{\left( 1\right) }\left(
x,y,f\right) +{S}_{nn}^{\left( 2\right) }\left( x,y,f\right) +{S}_{nn}%
^{\left( 3\right) }\left( x,y,f\right) .
\emd
It is clear, that
\begin{equation}\label{a17}
|S_{nn}^{(3)}(x,y,f)|\lesssim  \|f\|_{L^1(\ZT^2)} \lesssim
1+\iint\limits_{\mathbb{T}^{2}}|f|\log ^{+}|f| .
\end{equation}
Everywhere below the notation
\md0
\text{ p.v.}\iint_{\ZT^2}f(t,s)dtds
\emd
stands for either
\md0
\text{ p.v.}\int_{\ZT}\left(\text{p.v.}\int_\ZT f(t,s)dt\right)ds, \hbox { or } \text{ p.v.}\int_{\ZT}\left(\text{p.v.}\int_\ZT f(t,s)ds\right)dt
\emd
and in each cases we have equality of these two iterated integrals. To observe that we will need just the fact that $ f\in L\log L (\ZT)$ implies $\tilde f\in L_1(\ZT)$. Hence, making simple transformations and then changing the variables, we get
\md9\label{a14}
{S}_{nn}^{\left( 1\right) }&\left( x,y,f\right)  \\
&=\text{p.v.}\frac{1}{2\pi ^{2}}
\iint\limits_{\mathbb{T}^2}\frac{\sin n(t+s)}{2\tan (t/2)}
f\left( x+t,y+s\right) ds dt
\\
&+\text{p.v.}\frac{1}{2\pi ^{2}}
\iint\limits_{\mathbb{T}^2}\frac{\sin n (t-s)}{2\tan (t/2)}
 f\left( x+t,y+s\right) ds dt
\\
&=\text{p.v.}\frac{1}{2\pi ^{2}}
\iint\limits_{\mathbb{T}^2}\frac{\sin n u \cdot f\left( x+v,y+u-v\right)}{2\tan (v/2)}  dvdu\quad(u=t+s,\, v=t)
\\
&+\text{p.v.}\frac{1}{2\pi ^{2}}
\iint\limits_{\mathbb{T}^2}\frac{\sin n u \cdot f\left( x+v,y+v-u\right)}{2\tan (v/2)}  dvdu\quad(u=t-s,\, v=t)
\\
&=\frac{1}{2\pi }
\int\limits_{\mathbb{T}}\sin n u \left(\text{p.v.} \frac{1}{\pi }\int\limits_{\mathbb{T}} \frac{f\left( x+v,y+u-v\right)}{2\tan (v/2)}
dv\right)du
\\
&+\frac{1}{2\pi }
\int\limits_{\mathbb{T}}\sin n u \left(\text{p.v.} \frac{1}{\pi }\int\limits_{\mathbb{T}} \frac{f\left( x+v,y+v-u\right)}{2\tan (v/2)}
dv\right)du.
\emd
Observe, that the functions
\md8
&F_1(x,y,u)=\text{p.v.} \frac{1}{\pi }\int\limits_{\mathbb{T}} \frac{f\left( x+v,y+u-v\right)}{2\tan (v/2)}dv\\
&F_2(x,y,u)=\text{p.v.} \frac{1}{\pi }\int\limits_{\mathbb{T}} \frac{f\left( x+v,y+v-u\right)}{2\tan (v/2)}dv
\emd
are defined for almost all triples $(x,y,u)$. Moreover, we shall prove that
\md3\label{a15}
\iiint\limits_{\ZT^3}|F_i(x,y,u)|dxdydu\lesssim 1+\iint\limits_{\mathbb{T}^{2}}|f|\log ^{+}|f|, \quad i=1,2.
\emd
Consider the function $h(t,s,u):=f(t+s,t+u-s)$. Substituting $x=t+s$ and  $y=t-s$ in the expression of $F_1$, we get
\md0
F_1(t+s,t-s,u)=\text{p.v.} \frac{1}{\pi }\int\limits_{\mathbb{T}} \frac{h\left( t,s+v,u\right)}{2\tan (v/2)}dv.
\emd
Thus, first using the inequality \e{Zygmund} for variable $s$, then integrating by $t$ and $u$, we obtain
\md0
\iiint\limits_{\ZT^3}|F_1(t+s,t-s,u)|dsdtdu\lesssim 1+\iiint\limits_{\mathbb{T}^{3}}|h(t,s,u)||\log ^{+}|h(t,s,u)|dtdsdu.
\emd
After the change of variables $t=(x+y)/2$ and $s=(x-y)/2$ in the integrals, we get \e{a15} in the case $i=1$. The case $i=2$ may be proved similarly.
On the other hand, from \e{a14} it follows that
\md0
|S_{nn}^{(1)}\left( x,y,f\right)|\le \frac{1}{2\pi}\int\limits_\ZT|F_1(x,y,u)|du+\frac{1}{2\pi}\int\limits_\ZT|F_2(x,y,u)|du.
\emd
Combining this inequality with \e{a15}, we obtain
\md1\label{a16}
\iint\limits_{\ZT^2}\sup_n|S_{nn}^{(1)}\left( x,y,f\right)|dxdy\lesssim 1+\iint\limits_{\mathbb{T}^{2}}|f|\log ^{+}|f|
\emd
Similarly we can get the same bound for  $S_{nn}^{(2)}\left( x,y,f\right)$, which together with \e{a17} completes the proof of lemma.
\end{proof}

\section{Proof of Theorems}
\begin{proof}[Proof of Theorem \ref{BMO}]
From Lemma 2 we obtain
\md0
|S_{nn}(x,y,f)-S_{nn}^*(x,y,f)|\le \phi(x,y),\quad n=1,2,\ldots,
\emd
where the function $\phi(x,y)\ge 0$ satisfies the bound
\md0
\iint\limits_{\ZT^2}\phi(x,y)dxdy\lesssim 1+\iint\limits_{\mathbb{T}^{2}}|f|\log ^{+}|f|.
\emd
Thus we get
\md0
\bmo[S_{nn}(x,y,f)]\le \bmo[S_{nn}^*(x,y,f)]+3\phi(x,y).
\emd
Hence, the theorem will be proved, if we obtain $\bmo$ weak $(1,1)$ estimate for modified partial sums.  We have
\md9\label{a1}
S_{nn}^*\left( x,y,f\right) &  \\
=&\frac{1}{2\pi ^{2}}\iint\limits_{\mathbb{T}
^{2}}\frac{\cos n(t-s)\cdot f\left(
x+t,y+s\right)}{4\tan \left(
t/2\right) \tan \left( s/2\right) } dtds\notag \\
-&\frac{1}{2\pi ^{2}}\iint\limits_{\mathbb{T}
^{2}}\frac{\cos n(t+s)\cdot f\left(
x+t,y+s\right)}{4\tan \left(
t/2\right) \tan \left( s/2\right) } dtds\notag \\
=&\frac{1}{2\pi ^{2}}\iint\limits_{\mathbb{T}
^{2}}\frac{\cos nu \cdot f\left(
x+u+v,y+v\right)}{4\tan \left(
(u+v)/2\right) \tan \left( v/2\right) } dudv\quad (u=t-s,v=s)\notag \\
-&\frac{1}{2\pi ^{2}}\iint\limits_{\mathbb{T}
^{2}}\frac{\cos nu \cdot f\left(
x+u+v,y-v\right)}{4\tan \left(
(u+v)/2\right) \tan \left( v/2\right) } dudv\quad (u=t+s,v=-s)\notag\\
=& I_n(x,y,f)-J_n(x,y,f).\notag
\emd
Using a simple and an important identity
\md7
\frac{1}{\tan ((u+v)/2)\tan \left( v/2\right) }=\\
\frac{1}{\tan(u/2)\tan(v/2)}-\frac{1}{\tan(u/2)\tan ((u+v)/2)}-1,
\emd
we obtain
\md8
I_n(x,y,f)\\
=&\text{p.v.}\frac{1}{2\pi ^{2}}\iint\limits_{\mathbb{T}^2}\frac{\cos nu\cdot f\left(x+u+v,y+v\right)}{4\tan(u/2)\tan(v/2) } dudv\\
-&\text{p.v.}\frac{1}{2\pi ^{2}}\iint\limits_{\mathbb{T}^{2}}\frac{\cos nu\cdot f\left(x+u+v,y+v\right) }{4\tan(u/2)\tan((u+v)/2) }dudv\\
-&\frac{1}{2\pi ^{2}}\iint\limits_{\mathbb{T}^{2}}f\left(x+t,y+s\right) dtds\\
= &\text{p.v.}\frac{1}{2\pi }\int\limits_{\mathbb{T}}\frac{\cos nu}{2\tan(u/2) }\left(\text{p.v.} \frac{1}{\pi } \int\limits_{\mathbb{T}}\frac{f\left(x+u+v,y+v\right) } {2\tan(v/2 ) }dv\right)du\\
- &\text{p.v.}\frac{1}{2\pi }\int\limits_{\mathbb{T}}\frac{\cos nu}{2\tan(u/2) }\left(\text{p.v.}\frac{1}{\pi } \int\limits_{\mathbb{T}}\frac{f\left(x+u+v,y+v\right) } {2\tan((u+v)/2 ) }dv\right)du\\
-&\frac{1}{2\pi ^{2}}\iint\limits_{\mathbb{T}^{2}}f\left(t,s\right) dtds=I_n^{(1)}(x,y,f)-I_n^{(2)}(x,y,f)-I^{(0)},
\emd
where
\md1\label{a7}
|I^{(0)}|=\frac{1}{2\pi ^{2}}\left|\iint\limits_{\mathbb{T}^{2}}f\left(t,s\right) dtds\right|\lesssim 1+\iint\limits_{\mathbb{T}^{2}}|f(x,y)|\log ^{+}|f(x,y)|dxdy.
\emd
Observe that
\md0
I_n^{(1)}(x,y,f)=\frac{1}{2}\cdot U_n(x,A(\cdot, y))
\emd
where
\md0
A(x,y)=\text{p.v.}\frac{1}{\pi } \int\limits_{\mathbb{T}}\frac{f\left(x+v,y+v\right) } {2\tan(v/2 ) }dv.
\emd
Denoting $g(t,s):=f(t+s,t-s)$ and substituting $x=t+s$ and  $y=t-s$ we get
\md0
A(t+s,t-s)=\text{p.v.}\frac{1}{\pi } \int\limits_{\mathbb{T}}\frac{g\left(t+v,s\right) } {2\tan(v/2 ) }dv.
\emd
Using the inequality \e{Zygmund} for variable $t$ and then integrating by $s$, we obtain
\md0
\iint\limits_{\ZT^2}|A(t+s,t-s)|dsdt\lesssim 1+\iint\limits_{\mathbb{T}^{2}}|g(t,s)||\log ^{+}|g(t,s)|dtds.
\emd
After the changing back of variables $t=(x+y)/2$ and $s=(x-y)/2$ we get
\begin{equation}
\iint\limits_{\mathbb{T}^{2}}\left\vert A\left( x,y\right) \right\vert
dxdy\lesssim 1+\iint\limits_{\mathbb{T}^{2}}|f(x,y)|\log ^{+}|f(x,y)|dxdy.
\label{Zy1}
\end{equation}
Hence, applying the Lemma 1, we conclude
\md7\label{a4}
|\{(x,y)\in \mathbb{T}^2:\,\bmo[I_n^{(1)}(x,y,f)]>\lambda \}|\\
\lesssim \frac{1}{%
\lambda }\left(1+\iint\limits_{\mathbb{T}^{2}}|f(x,y)|\log ^{+}|f(x,y)|dxdy\right).
\emd
After the changing of variable $u+v\to \nu$ in the inner integral of the expression of $I_n^{(2)}(x,y,f)$ we get
\md0
I_n^{(2)}(x,y,f)=\text{p.v.}\frac{1}{2\pi }\int\limits_{\mathbb{T}}\frac{\cos nu}{2\tan(u/2) }\left(\text{p.v.}\frac{1}{\pi } \int\limits_{\mathbb{T}}\frac{f\left(x+\nu,y+\nu-u\right) } {2\tan(\nu/2 ) }d\nu\right)du,
\emd
and then analogously we can prove that
\md7\label{a5}
|\{(x,y)\in \mathbb{T}^2:\,\bmo[I_n^{(2)}(x,y,f)]>\lambda \}|\\
\lesssim \frac{1}{%
\lambda }\left(1+\iint\limits_{\mathbb{T}^{2}}|f(x,y)|\log ^{+}|f(x,y)|dxdy\right).
\emd
Hence, using \e{a7}, \e{a4} and \e{a5}, we obtain
\md6
|\{(x,y)\in \mathbb{T}^2:\,\bmo[I_n(x,y,f)]>\lambda \}|\\
\lesssim \frac{1}{%
\lambda }\left(1+\iint\limits_{\mathbb{T}^{2}}|f(x,y)|\log ^{+}|f(x,y)|dxdy\right).
\emd
Using the absolutely same process we may get the analogous estimate for $J_n(x,y,f)$ and therefore  for $S_{nn}^*(x,y,f)$.
The theorem is proved.
\end{proof}
Let $X$ be either $[0,1]$ or $\ZT^2$ and $L_{M}=L_{M}(X)$ is the Orlicz space of functions on $X$, generated by Young function $M$,
 i. e. $M$ is convex continuous even function such that $M\left( 0\right)
=0 $ and%
\begin{equation*}
\lim_{t\rightarrow 0+}\frac{M(t)}{t}=\lim_{t\rightarrow \infty }\frac{t}{M(t)%
}=0.
\end{equation*}%
It is well known that $L_{M}$ is a Banach space with respect to Luxemburg norm%
\begin{equation*}
\left\Vert f\right\Vert _{(M)}:=\inf \left\{ \lambda :\,\lambda
>0,\,\int\limits_X M\left( \frac{|f|}{\lambda }\right) \leq
1\right\} <\infty .
\end{equation*}
 We will need some basic
properties of Orlicz spaces (see \cite{KrRu} ).

1) According to a theorem from (\cite{KrRu}, chap. 2, theorem 9.5) we have%
\begin{equation}  \label{d1}
\left\Vert f\right\Vert _{(M)}\leq 1\Rightarrow \int\limits_X
M\left( |f|\right) \leq \left\Vert f\right\Vert _{(M)},
\end{equation}%

2) From this fact we may deduce, that%
\begin{equation}  \label{d3}
0,5\left( 1+\int\limits_XM\left( |f|\right) \right) \leq
\left\Vert f\right\Vert _{(M)}\leq  1+\int\limits_X
M\left( |f|\right)
\end{equation}%
provided $\left\Vert f\right\Vert _{(M)}=1$.

3) From the definition of norm $\|\cdot\|_{(M)}$ immediately follows that $|f(x)|\le |g(x)|$ implies $\|f\|_{(M)}\le \|g\|_{(M)}$. Besides, for any measurable set $E$ we have%
\begin{equation*}
\left\Vert \mathbb{I}_{E}\right\Vert _{(M)}=o\left( 1\right) \text{ as }%
\left\vert E\right\vert \rightarrow 0\text{ \ (\cite{KrRu}, (9.23)).}
\end{equation*}

4) If $M$ satisfies $\Delta _{2}$-condition, that is%
\begin{equation*}
M\left( 2t\right) \le cM\left( t\right) ,t>t_{0},
\end{equation*}%
and $X=\ZT^2$, then the set of two variable trigonometric polynomials on $\ZT^2$ is dense in $L_{M}$ (\cite{KrRu}, \S 10).

5) From (\ref{d1}) it follows that for any sequence of functions $f_{n}$ the
condition $\left\Vert f_{n}\right\Vert _{(M)}\rightarrow 0$ implies $%
\int\limits_XM\left( |f_{n}|\right) \rightarrow 0.$

\begin{proof}[Proof of \trm{exp}]
We will deal with two $M$-functions
\md4
\Phi(t)=t\log^+t,\\
\Psi(t)=\exp t-1.
\emd
We consider two Orlicz spaces $L_\Phi= L_\Phi(\ZT^2)$ and  $L_\Psi= L_\Psi(0,1)$. Combining \e{d3} with \trm{BMO}, we may obtain
\md1\label{d4}
|\{ ( x,y) \in \mathbb{T}^{2}:\,\bmo [S_{nn}(x,y,f)] >\lambda \}
\lesssim \frac{\|f\|_{(\Phi)}}{\lambda}.
\emd
Indeed, at first we deduce the case when $\|f\|_{(\Phi)}= 1$, then, using a linearity principle, we get the inequality in the general case.

The inequality
\md1\label{d5}
\|f\|_{(\Psi)} \lesssim \|f\|_{\bmo}
\emd
proved in \cite{Ro}. It is an immediate consequence of the John-Nirenberg theorem.
Denote
\md1\label{d6}
\ZB f(x,y)=\sup_{0\le n<\infty}\left\|\sum_{k=0}^nS_{kk}(x,y,f)\ZI_{\delta_k^n}(t)\right\|_{{(\Psi)}}.
\emd
Notice, that by the definition we have
\md0
\bmo [S_{nn}(f,x,y)]=\sup_{0\le n<\infty}\left\|\sum_{k=0}^nS_{kk}(x,y,f)\ZI_{\delta_k^n}(t)\right\|_{\bmo}.
\emd
So, taking into account  \e{d4} and \e{d5} we obtain
\md1\label{d7}
|\{ ( x,y) \in \mathbb{T}^{2}:\,\ZB f(x,y) >\lambda \}
\lesssim \frac{\|f\|_{(\Phi)}}{\lambda}.
\emd
On the other hand we have
\md6
\frac{1}{n+1}\sum_{k=0}^{n}(
\exp A|S_{kk}( x,y,f) -f( x,y)| -1) \\
=\frac{1}{n+1}\sum_{k=0}^n\Psi(A|S_{kk}( x,y,f) -f( x,y)|)\\
=\int_0^1\Psi\left(A\sum_{k=0}^n|S_{kk}(x,y,f)-f(x,y)|\ZI_{\delta_k^n}(t)\right)dt.
\emd
Thus, according the property 5) of Orlicz spaces, to prove the theorem it is enough to prove that
\md1\label{d8}
\left\|\sum_{k=0}^n(S_{kk}(x,y,f)-f(x,y))\ZI_{\delta_k^n}(t)\right\|_{{(\Psi)}}\to 0,
\emd
almost everywhere on $\ZT^2$ as $n\to\infty$, for any $f\in L_\Phi$. It is easy to observe, that \e{d8} holds if $f$ is a real trigonometric polynomial in two variables. Indeed, if $P(x,y)$ is a polynomial of degree $m$, then we have
\md0
S_{kk}(x,y,P)-P(x,y)\equiv 0,\quad  k\ge m.
\emd
Therefore, if $n\ge m$, then we get
\md0
\left|\sum_{k=0}^n(S_{kk}(x,y,P)-P(x,y))\ZI_{\delta_k^n}(t)\right|\le C\cdot\ZI_{[0,m/(n+1)]}(t),
\emd
where $C$ is a constant, depending on $P$. Then, applying the property 3) of Orlicz spaces, we conclude that \e{d8} holds if $f=P$. To prove the general case, we consider the set
\md7\label{d9}
G_\lambda=\{(x,y)\in\ZT^2:\\ \limsup_{n\to\infty}\left\|\sum_{k=0}^n(S_{kk}(x,y,f)-f(x,y))\ZI_{\delta_k^n}(t)\right\|_{{(\Psi)}}>\lambda\}.
\emd
To complete the proof of theorem, it enough to prove that $|G_\lambda|=0$ if $\lambda>0$. It is easy to check that $\Phi(t)$ satisfies the $\Delta_2$-condition. Therefore, according the property 4), we may chose a polynomial $P(x,y)$ such that
$\|f-P\|_{(\Phi)}<\varepsilon$. Using the definition of $(\Phi)$-norm, we get
\md0
\int\limits_{\ZT^2}\Phi\left(\left|\frac{f-P}{\varepsilon}\right|\right)<1.
\emd
From Chebishev's inequality, one can easily deduce
\md0
|\{(x,y)\in\ZT^2:\,|f(x,y)-P(x,y)|>\lambda\}|\le \frac{1}{\Phi(\lambda/\varepsilon)},\quad \lambda>0.
\emd
Thus, using \e{d7} for any $\lambda>0$ we get
\md8
|G&_\lambda|=|\{(x,y)\in\ZT^2:\\
&\limsup_{n\to\infty}\left\|\sum_{k=0}^n(S_{kk}(x,y,f-P)-f(x,y)+P(x,y))\ZI_{\delta_k^n}(dt)\right\|_{{(\Psi)}}>\lambda\}|\\
&\le |\{ \ZB (f-P)(x,y)+c|f(x,y)-P(x,y)|>\lambda\}|\\
&\lesssim \frac{\|f-P\|_{(\Phi)}}{\lambda}+\frac{1}{\Phi(\lambda/\varepsilon)}\le \frac{\varepsilon}{\lambda}+\frac{1}{\Phi(\lambda/\varepsilon)}.
\emd
Since $\varepsilon>0$ may be taken sufficiently small, we conclude $|G_\lambda|=0$ if $\lambda>0$.
\end{proof}

{\bf Acknowledgement}. The authors would like to thank the referees for helpful suggestions.

\end{document}